\documentclass[11pt]{article}
\usepackage{amsmath}
\usepackage{amssymb}
\usepackage{amsthm}
\usepackage[usenames]{color}
\usepackage{amscd}
\usepackage{dsfont}
\usepackage{indentfirst}

\usepackage[colorlinks=true,linkcolor=blue,filecolor=red,
citecolor=webgreen]{hyperref}
\definecolor{webgreen}{rgb}{0,.5,0}

\numberwithin{equation}{section}

\def\C{{\mathds{C}}}

\def\R{{\mathbb{R}}}
\def\N{{\mathds{N}}}
\def\Z{{\mathds{Z}}}

\def\Res{\operatorname{Res}}
\def\1{{\bf 1}}

\newtheorem{theorem}{Theorem}[section]

\newtheorem{lemma}[theorem]{Lemma}
\newtheorem{cor}[theorem]{Corollary}

\newtheorem{definition}[theorem]{Definition}

\begin{document}

\title{{\bf Hyperbolic summation for functions of the GCD and LCM of several integers}}
\author{Randell Heyman \\
School of Mathematics and Statistics \\
University of New South Wales \\
Sydney, Australia \\
E-mail: {\tt randell@unsw.edu.au} \\ \ \\
L\'aszl\'o T\'oth \\ Department of Mathematics \\ University of P\'ecs \\
Ifj\'us\'ag \'utja 6, 7624 P\'ecs, Hungary \\ E-mail: {\tt ltoth@gamma.ttk.pte.hu}}
\date{}
\maketitle

\centerline{The Ramanujan Journal {\bf 62} (2023), 273--290}

\begin{abstract} Let $k\ge 2$ be a fixed integer. We consider sums of type $\sum_{n_1\cdots n_k\le x} F(n_1,\ldots,n_k)$,
taken over the hyperbolic region $\{(n_1,\ldots,n_k)\in \N^k: n_1\cdots n_k\le x\}$, where $F:\N^k\to \C$ is a 
given function. In particular, we deduce asymptotic formulas with remainder terms for the hyperbolic
summations $\sum_{n_1\cdots n_k\le x} f((n_1,\ldots,n_k))$ and $\sum_{n_1\cdots n_k\le x} f([n_1,\ldots,n_k])$,
involving the GCD and LCM of the integers $n_1,\ldots,n_k$, where $f:\N\to \C$ belongs to certain classes of functions.
Some of our results generalize those obtained by the authors \cite{HeyTot2021} for $k=2$.
\end{abstract}

{\sl 2010 Mathematics Subject Classification}: 11A05, 11A25, 11N37

{\sl Key Words and Phrases}: arithmetic function of several variables, convolute, greatest common divisor, least common multiple, 
hyperbolic summation, asymptotic formula, Piltz divisor problem

\section{Introduction} 

Let $\N:=\{1,2,\ldots\}$, let $F:\N^k\to \C$ be a function of $k$ ($k\ge 2$) variables, and 
consider the convolute of $F$, defined as the one variable function $\widetilde{F}: \N \to \C$ given by
\begin{equation*}
\widetilde{F}(n):= \sum_{n_1\cdots n_k=n} F(n_1,\ldots,n_k),
\end{equation*}
where the sum is over all $(n_1,\ldots,n_k)\in \N^k$ such that $n_1\cdots n_k=n$. Note that if $F$ is multiplicative, 
then $\widetilde{F}$ is also multiplicative. See 
Vai\-dy\-anathas\-wamy \cite{Vai1931}, T\'oth \cite[Sect.\ 6]{Tot2014}.

In this paper we look at sums of type  
\begin{equation*}
\sum_{n\le x} \widetilde{F}(n) = \sum_{n_1\cdots n_k\le x} F(n_1,\ldots,n_k),
\end{equation*}
taken over the hyperbolic region $\{(n_1,\ldots,n_k)\in \N^k: n_1\cdots n_k\le x\}$. In particular, 
given an arithmetic function $f:\N \to \C$, we  
are interested in the convolutes 
\begin{equation*} \label{G_f_k}
G_{f,k}(n):= \sum_{n_1\cdots n_k=n} f((n_1,\ldots,n_k)),
\end{equation*}
\begin{equation*} \label{L_f_k}
L_{f,k}(n):= \sum_{n_1\cdots n_k=n} f([n_1,\ldots,n_k]),
\end{equation*}
involving the GCD and LCM of integers. If $f$ is multiplicative, then the functions 
$G_{f,k}$ and $L_{f,k}$ are multiplicative as well.

Asymptotic formulas for sums 
\begin{equation*}
\sum_{n\le x} G_{f,k}(n) = \sum_{n_1\cdots n_k\le x} f((n_1,\ldots,n_k)),
\end{equation*}
in the case of certain special functions 
$f$ and for $k\ge 2$, in particular for $k=2$, were given by Heyman \cite{Hey2020}, Heyman and T\'oth 
\cite{HeyTot2021}, Kiuchi and Saad Eddin \cite{KS2021}, Kr\"atzel, Nowak and T\'oth \cite{KNT2012}. 
Some related probabilistic properties were studied by Iksanov, Marynych and Raschel \cite{IMR2022}. 

In fact, for every function $f$ one has 
\begin{equation} \label{G_form}
G_{f,k}(n)= \sum_{d^k\delta=n} (\mu*f)(d)\tau_k(\delta),
\end{equation}
where $\mu$ is the M\"obius function, $\tau_k$ is the $k$-factors Piltz divisor function
and $*$ denotes the convolution of arithmetic functions. See \cite[Prop.\ 5.1]{KNT2012}. 
Identity \eqref{G_form} shows that asymptotic formulas for the sums $\sum_{n\le x} G_{f,k}(n)$ are closely
related to asymptotics for the Piltz divisor function. 

Moreover, in the case of certain pairs $(f,k)$, asymptotic formulas for the sums $\sum_{n\le x} G_{f,k}(n)$ 
reduce to the Piltz divisor problem. For example, let $f(n)=n$ ($n\in \N$) and $k\ge 4$. Then, as mentioned in 
\cite[Sect.\ 4]{KNT2012}, it follows by elementary convolution arguments that if $\theta_k\ge 1/2$ is any real number such that
\begin{equation} \label{Piltz_form} 
\sum_{n\le x} \tau_k(n) = x\, P_{k-1}(\log x) + O(x^{\theta_k+\varepsilon})
\end{equation} 
holds for every $\varepsilon>0$, where $P_{k-1}(\log x)= \underset{s=1}\Res \left( \zeta^k(s)\frac{x^{s-1}}{s}\right)$ is a polynomial
in $\log x$ of degree $k-1$, with leading coefficient $1/(k-1)!$, then also 
\begin{equation}  \label{sum_gcd_k}
\sum_{n_1\cdots n_k\le x} (n_1,\ldots,n_k) = x\, Q_{k-1}(\log x) + O(x^{\theta_k+\varepsilon}),
\end{equation}
where $k\ge 4$, $Q_{k-1}(\log x) = \underset{s=1}\Res \left(\frac{\zeta^k(s)\zeta(ks-1)}{\zeta(ks)} \frac{x^{s-1}}{s}\right)$
is another polynomial in $\log x$ of degree $k-1$, with leading coefficient $\zeta(k-1)/((k-1)!\zeta(k))$. Note that here one 
can choose, e.g., $\theta_2=1/2$ and $\theta_k=\frac{k-1}{k+1}$ ($k\ge 3$), see Titchmarsh \cite[Th.\ 12.2]{Tit1986}. 
Also see Bordell\`es \cite[Sect.\ 4.7.6]{Bor2020}.

The error terms corresponding to $f(n)=n$ with $k=2$ and $k=3$ were investigated in \cite{KNT2012}, by analytic methods. For 
$f(n)=\tau_2(n)=:\tau(n)$ and $k=2$ see \cite{Hey2020}, \cite{HeyTot2021}. For $k\ge 3$ the cases of the divisor function $\tau(n)$ and the M\"obius function $\mu(n)$ were studied in \cite{KS2021}, giving explicit error terms, and computing the main terms for 
$k=3$ and $k=4$. However, note that for $f(n)=\tau(n)$, by \eqref{G_form},
\begin{equation*} 
\sum_{a_1\cdots a_k=n} \tau((a_1,\ldots,a_k)) = \sum_{d^k\delta = n} \tau_k(\delta)= 
\sum_{a_1\cdots a_kd^k=n} 1 =: \tau(\underbrace{1,\ldots,1}_{k},k)(n),
\end{equation*}
and the summation of this divisor function is known in the literature. See, e.g., Kr\"atzel \cite{Kra1988}.

In \cite{HeyTot2021} we established asymptotic formulas for $\sum_{n\le x} G_{f,2}(n)$ (in the case $k=2$) for various classes of functions $f$ by elementary arguments, in particular for the functions $f(n)= \log n, \omega(n), \Omega(n)$. In this paper we extend some of these results for any $k\ge 2$.

In the case of the LCM there is no known formula similar to \eqref{G_form} and to
give asymptotics for $\sum_{n\le x} L_{f,k}(n)$, with good error terms, is more difficult. In \cite{HeyTot2021}
we considered the case $k=2$ and the functions $f(n)=n, \log n, \omega(n), \Omega(n), \tau(n)$. For example, we proved (see
\cite[Th.\ 2.11]{HeyTot2021}) that
\begin{equation*}
\sum_{mn\le x} \tau([m,n])= x P_3(\log x) + O(x^{1/2+\varepsilon}),    
\end{equation*}
where $P_3(t)$ is a polynomial in $t$ of degree $3$ with leading coefficient
\begin{equation*}
\frac1{\pi^2} \prod_p \left(1-\frac1{(p+1)^2}\right).    
\end{equation*}

In this paper we give asymptotic formulas for $\sum_{n\le x} L_{f,k}(n)$ with any $k\ge 2$ in the case 
of some classes of functions $f$. Our main results are presented in Section \ref{Section_Main_results}, and
their proofs are included in Section \ref{Section_Proofs}.

For some different asymptotic results concerning functions of the GCD and LCM of several integers we refer to
Bordell\`es and T\'oth \cite{BorTot2022}, Hilberdink and T\'oth \cite{HilTot2016}, T\'oth and Zhai \cite{TotZha2018}, and their references.
For summations over $mn\le x$ of certain other two variables functions $F(m,n)$ see Kiuchi and Saad Eddin \cite{KS2020}, 
Sui and Liu \cite{SL2020}.

Throughout the paper we use the following notation: $\N =\{1,2,\ldots\}$; $(n_1,\ldots,n_k)$ and $[n_1,\ldots,n_k]$ denote the greatest common divisor (GCD) and least common multiple (LCM) of $n_1,\ldots,n_k\in \N$; $\varphi$ is Euler's totient function; 
$\tau(n)$ and $\sigma(n)$ are the number and sum of divisors of $n\in \N$; $\tau_k$ is the $k$-factors Piltz divisor function;  
$\omega(n)$ and $\Omega(n)$ stand for the number of prime divisors, respectively prime power divisors of $n\in \N$;
$*$ is the Dirichlet convolution of arithmetic functions of $k$ variables; 
$\mu$ denotes the M\"obius function of $k$ variables; the sums $\sum_p$ and products $\prod_p$ are taken over the primes $p$.

\section{Main results} \label{Section_Main_results}

For functions $F,G:\N^k\to \C$ ($k\ge 1$) consider their convolution $F*G$ defined by
\begin{equation} \label{convo_k}
(F*G)(n_1,\ldots,n_k)=\sum_{d_1\mid n_1, \ldots, d_k\mid n_k} F(d_1,\ldots,d_k) G(n_1/d_1,\ldots, n_k/d_k), 
\end{equation}
and the generalized M\"obius function $\mu(n_1,\ldots,n_k) = \mu(n_1)\cdots \mu(n_k)$, which is the inverse of the $k$-variable 
constant $1$ function under convolution \eqref{convo_k}. See the survey \cite{Tot2014} on properties of (multiplicative) arithmetic 
functions of several variables. 

Our first result is the following.
\begin{theorem} \label{Th_Wintner_F} Let $F:\N^k \to \C$ be an arbitrary arithmetic function of $k$ variables 
\textup{($k\ge 1$)}, and assume that the multiple series
\begin{equation} \label{series_F}
\sum_{n_1,\ldots,n_k=1}^{\infty} \frac{(\mu*F)(n_1,\ldots,n_k)}{n_1\cdots n_k}
\end{equation}
is absolutely convergent. Then 
\begin{equation*}
\lim_{x\to \infty} \frac1{x(\log x)^{k-1}}  \sum_{n_1\cdots n_k \le x} F(n_1,\ldots,n_k) = 
\frac1{(k-1)!} C_{F,k},
\end{equation*}
where $C_{F,k}$ is the sum of series \eqref{series_F}.
\end{theorem}

For $k=1$ this is Wintner's mean value theorem, going back to the work of van der Corput.  See, e.g., \cite{Coh1961}, 
\cite[Th.\ 2.19]{Hil}, \cite[p.\ 138]{Pos1988}.
Also, Theorem \ref{Th_Wintner_F} is the analog of the corresponding result for summation of functions $F(n_1,\ldots,n_k)$ with $n_1,\ldots,n_k\le x$, 
obtained by Ushiroya \cite{Ush2012}. Note that if $F$ is multiplicative, then
\begin{equation*}
C_{F,k} = \prod_p \left(1-\frac1{p}\right)^k \sum_{\nu_1,\ldots,\nu_k=0}^{\infty} \frac{F(p^{\nu_1},\ldots,p^{\nu_k})}{p^{\nu_1+\cdots +\nu_k}}.    
\end{equation*}

If $F(n_1,\ldots,n_k)=f((n_1,\ldots,n_k))$, then we deduce the next result.

\begin{theorem} \label{Cor_Wintner_f} Let $f:\N \to \C$ be an arbitrary arithmetic function of one variable, let $k\ge 2$ and assume that the 
series
\begin{equation*} 
\sum_{n=1}^{\infty} \frac{f(n)}{n^k}
\end{equation*}
is absolutely convergent. Then 
\begin{equation*}
\lim_{x\to \infty} \frac1{x(\log x)^{k-1}}  \sum_{n_1\cdots n_k \le x} f((n_1,\ldots,n_k)) = 
\frac1{(k-1)!\zeta(k)} \sum_{n=1}^{\infty} \frac{f(n)}{n^k}. 
\end{equation*}
\end{theorem}

For example, taking the function $f(n)=\log n$ we deduce that for every $k\ge 2$ one has 
\begin{equation*}
\prod_{n_1\cdots n_k\le x} (n_1,\ldots,n_k) = x^{(1+o(1))\frac{K}{(k-1)!} x (\log  x)^{k-2}}, \quad \text{ as $x\to \infty$}, 
\end{equation*}
where the constant $K:=K_{\log,k}$ is given by \eqref{K}. For $k\ge 3$ we obtain more precise formulas for the functions $\log n, 
\omega(n), \Omega(n)$, included in the next theorem. The case $k=2$ has been discussed by the authors \cite[Cor.\ 2.6]{HeyTot2021}.

\begin{theorem} \label{Th_gcd} Let $k\ge 3$ and let $f$ be one of the functions $\log n, \omega(n), \Omega(n)$. Let $\theta_k\ge 1/k$ 
denote any real number satisfying \eqref{Piltz_form}. Then
\begin{equation*}
\sum_{n_1\cdots n_k\le x} f((n_1,\ldots,n_k)) = x\, P_{f,k-1}(\log x) +O(x^{\theta_k+\varepsilon})    
\end{equation*}
where $P_{f,k-1}(t)$ are polynomials in $t$ of degree $k-1$ with leading coefficient $\frac1{(k-1)!} K_{f,k}$, and where 
\begin{equation} \label{K}
K_{\log,k}= \sum_p \frac{\log p}{p^k-1}, \quad
K_{\omega,k}= \sum_p \frac1{p^k}, \quad 
K_{\Omega,k}= \sum_p \frac1{p^k-1}.
\end{equation}
\end{theorem}

The following class of functions was defined by Hilberdink and T\'oth \cite{HilTot2016}. 

\begin{definition} \label{Def_A_r}
Given  a fixed real number $r$ let ${\cal A}_r$ denote the
class of multiplicative arithmetic functions $f:\N \to \C$
satisfying the following properties: there exist real constants
$C_1,C_2$ such that $|f(p)-p^r|\le C_1\, p^{r-1/2}$ for every prime $p$, and 
$|f(p^{\nu})|\le C_2\, p^{\nu r}$ for every prime power $p^\nu$ with $\nu \ge 2$.
\end{definition}

Observe that the functions $f(n)=n^r$, $\sigma(n)^r$, $\varphi(n)^r$
belong to the class ${\cal A}_r$ for every $r\in \R$.
See \cite{HilTot2016} for some other examples of functions in class ${\cal A}_r$, including sums of 
divisor functions (both standard and alternating) and generalisations of both Euler and Dedekind functions.

The following result was proved in \cite[Th.\ 2.1]{HilTot2016}. Let $k\ge 2$ be a fixed integer and 
let $f\in {\cal A}_r$ be a function, where $r\ge 0$ is
real. Then for every $\varepsilon >0$,
\begin{equation} \label{form}
\sum_{n_1,\ldots,n_k\le x} f([n_1,\ldots,n_k]) = C_{f,k}
\frac{x^{k(r+1)}}{(r+1)^k}
+ O\big(x^{k(r+1)-\frac1{2} +\varepsilon}\big),
\end{equation}
where
\begin{equation} \label{C_f_k} 
C_{f,k}= \prod_p \left(1-\frac1{p}\right)^k \sum_{\nu_1,\ldots,
\nu_k=0}^{\infty}
\frac{f(p^{\max(\nu_1,\ldots,\nu_k)})}{p^{(r+1)(\nu_1 +\cdots
+\nu_k)}}.
\end{equation}

In this paper we prove the following related result.

\begin{theorem} \label{Th_class_r} Let $k\ge 2$ be a fixed integer and let $f$ be a function in the class ${\cal A}_r$, 
given by Definition \ref{Def_A_r}, where $r\ge 0$ is real. Let $\theta_k\ge 1/2$ be any real number satisfying 
\eqref{Piltz_form}.  Then for every $\varepsilon >0$,
\begin{equation} \label{f_form}
\sum_{n_1\cdots n_k\le x} f([n_1,\ldots,n_k]) = x^{r+1} Q_{f,k-1} (\log x)
+ O\big(x^{r+\theta_k+\varepsilon}\big),
\end{equation}
where $Q_{f,k-1}(t)$ is a polynomial in $t$ of degree $k-1$ with leading coefficient $\frac1{(r+1)(k-1)!}C_{f,k}$, the
constant $C_{f,k}$ being given by \eqref{C_f_k}.
\end{theorem}
 
We point out the next formula, which is the counterpart of \eqref{sum_gcd_k}.

\begin{cor} {\rm ($f(n)=n$, $r=1$)} Let $k\ge 2$. Then for every $\varepsilon >0$,
\begin{equation*} 
\sum_{n_1\cdots n_k\le x} [n_1,\ldots,n_k] = x^2 \overline{Q}_{k-1} (\log x)
+ O\big(x^{3/2+\varepsilon}\big),
\end{equation*}
where $\overline{Q}_{k-1}(t)$ is a polynomial in $t$ of degree $k-1$ with leading coefficient $\frac1{2(k-1)!}C_k$, and
\begin{equation*}  
C_k= \prod_p \left(1-\frac1{p}\right)^k \sum_{\nu_1,\ldots,
\nu_k=0}^{\infty} \frac1{p^{2(\nu_1 +\cdots
+\nu_k)-\max(\nu_1,\ldots,\nu_k)}}.
\end{equation*}
\end{cor}

Note that for $k=2$ this was proved by Heyman and T\'oth \cite[Th.\ 2.7]{HeyTot2021} with a better error term, and one has $C_2=\zeta(3)/\zeta(2)$. 

Theorem \ref{Th_class_r} does not apply for the divisor function $\tau(n)$, and we prove the next result.

\begin{theorem} \label{Th_tau} Let $k\ge 2$ be a fixed integer. Then for every $\varepsilon >0$,
\begin{equation} \label{tau_form}
\sum_{n_1\cdots n_k\le x} \tau([n_1,\ldots,n_k]) = x\, Q_{\tau,2k-1} (\log x)
+ O\big(x^{\theta_{2k}+\varepsilon}\big),
\end{equation}
where $Q_{\tau,2k-1}(t)$ is a polynomial in $t$ of degree $2k-1$ with leading coefficient $D_k/(2k-1)!$, the
constant $D_k$ given by 
\begin{equation*} 
D_k= \prod_p \left(1-\frac1{p}\right)^{2k} \sum_{\nu_1,\ldots,
\nu_k=0}^{\infty} \frac{\max(\nu_1,\ldots,\nu_k)+1}{p^{\nu_1 +\cdots +\nu_k}},
\end{equation*}
and $\theta_{2k}\ge 1/2$ being any exponent in the $2k$-factor 
Piltz divisor problem. In particular, one can select $\theta_{2k} = \frac{2k-1}{2k+1}$ \textup{($k\ge 2$)}. 
\end{theorem}
 
\section{Proofs} \label{Section_Proofs}

Theorem \ref{Th_Wintner_F} is, in fact, a special case of the next result due to Cohen \cite{Coh1961}.

\begin{lemma} \label{Th_Cohen}  Let $f:\N \to \C$ be an arbitrary arithmetic function, let $k\ge 1$, and write 
$f(n)=\sum_{d\delta=n} g(d)\tau_k(\delta)$ \textup{($n\in \N$)}. If
the series $\sum_{n=1}^{\infty} \frac{g(n)}{n}$ is absolutely convergent, then  
\begin{equation*}
\lim_{x\to \infty} \frac1{x(\log x)^{k-1}}  \sum_{n\le x} f(n) = 
\frac1{(k-1)!} \sum_{n=1}^{\infty} \frac{g(n)}{n}.
\end{equation*}
\end{lemma}

Note that here $g=f*\mu*\cdots *\mu$, taking $k$-times the function $\mu$.

\begin{proof}[Proof of Theorem {\rm \ref{Th_Wintner_F}}] Apply Lemma \ref{Th_Cohen}. Given an arbitrary function 
$F:\N^k \to \C$, choose $f(n)=\widetilde{F}(n):=\sum_{n_1\cdots n_k} F(n_1,\ldots,n_k)$. Then
\begin{equation*}
\widetilde{F}(n)= \sum_{n_1\cdots n_k=n} \sum_{d_1\mid n_1,\ldots,d_k\mid n_k} (\mu*F)(d_1,\ldots,d_k)   
=  \sum_{d_1\delta_1\cdots d_k\delta_k=n} (\mu*F)(d_1,\ldots,d_k),
\end{equation*}
that is,
\begin{equation} \label{F}
\widetilde{F}(n) = \sum_{d\delta=n} g(d) \tau_k(\delta),   
\end{equation}
where 
\begin{equation*} 
g(d) = \sum_{d_1\cdots d_k=d} (\mu*F)(d_1,\ldots,d_k),
\end{equation*}
and
\begin{equation} \label{series_g}
\sum_{n=1}^{\infty} \frac{g(n)}{n} =  \sum_{n_1,\ldots,n_k=1}^{\infty} \frac{(\mu*F)(n_1,\ldots,n_k)}{n_1\cdots n_k},
\end{equation}
finishing the proof.
\end{proof}

For the sake of completeness we also present a proof of Lemma \ref{Th_Cohen}, which is different from the proofs
given by Cohen \cite{Coh1961} and Narkiewicz \cite{Nar1963}.

\begin{proof}[Proof of Lemma {\rm \ref{Th_Cohen}}] Assume that $k\ge 2$. We use the estimate 
\begin{equation} \label{element_est_tau_k}
\sum_{n\le x} \tau_k(n)= \frac1{(k-1)!}x (\log x)^{k-1} + O(x(\log x)^{k-2}),
\end{equation}
see, e.g., Nathanson \cite[Th.\ 7.6]{Nat2000} for an elementary proof by induction on $k$.
Note that this is sufficient here, the complete formula \eqref{Piltz_form} is not needed.

According to \eqref{element_est_tau_k},
\begin{equation*}
S_f(x):= \sum_{n\le x} f(n) = \sum_{d\le x} g(d) \sum_{\delta \le x/d} \tau_k(\delta)    
\end{equation*}
\begin{equation*}
 = \frac{x}{(k-1)!} \sum_{d\le x} \frac{g(d)}{d} \left(\log \frac{x}{d}\right)^{k-1} + O\Big(x(\log x)^{k-2}
 \sum_{d\le x} \frac{|g(d)|}{d}\Big),    
\end{equation*}
where the $O$-term is $O(x(\log x)^{k-2})$ by the absolute convergence of the series $\sum_{n=1}^{\infty} \frac{g(n)}{n}$.

We deduce that
\begin{equation} \label{S_f_x}
\frac{(k-1)! S_f(x)}{x(\log x)^{k-1}} =  \sum_{d\le x} \frac{g(d)}{d} +  \sum_{j=1}^{k-1}
(-1)^j \binom{k-1}{j} (\log x)^{-j} \sum_{d\le x} \frac{g(d)}{d} (\log d)^j +
 O\left((\log x)^{-1}\right).
\end{equation}

Now for every $j$ ($1\le j \le k-1$) and for a small $\varepsilon >0$ we split the following sum in two parts:
\begin{equation*}
\sum_{d\le x} \frac{|g(d)|}{d} (\log d)^j  =  \sum_{d\le x^{\varepsilon}} \frac{|g(d)|}{d} (\log d)^j +
\sum_{x^{\varepsilon} < d\le x} \frac{|g(d)|}{d} (\log d)^j
\end{equation*}
\begin{equation*}
\le (\varepsilon \log x)^j \sum_{d=1}^{\infty} \frac{|g(d)|}{d}  + (\log x)^j  
\sum_{x^{\varepsilon} < d } \frac{|g(d)|}{d}.
\end{equation*}

Hence
\begin{equation*}
(\log x)^{-j} \sum_{d\le x} \frac{|g(d)|}{d} (\log d)^j  \le \varepsilon^j \sum_{d=1}^{\infty} \frac{|g(d)|}{d}  +  
\sum_{x^{\varepsilon} < d } \frac{|g(d)|}{d},
\end{equation*}
where the first term is arbitrary small if $\varepsilon$ is small, and the second term is also arbitrary small if $x$ is large enough (by the definition of convergent series). Now \eqref{S_f_x} shows that
\begin{equation*} 
\lim_{x\to \infty} \frac{(k-1)! S_f(x)}{x(\log x)^{k-1}} =  \sum_{d=1}^{\infty} \frac{g(d)}{d},
\end{equation*}
and the proof is complete.
\end{proof}

\begin{proof}[Proof of Theorem {\rm \ref{Cor_Wintner_f}}]
If $F(n_1,\ldots,n_k)= f((n_1,\ldots,n_k))$, then \eqref{F} and \eqref{G_form} show that
\begin{equation*}
\widetilde{F}(n)= \sum_{d\delta=n} g(d)\tau_k(\delta), 
\end{equation*}
where
\begin{equation*}
g(n)= \begin{cases} (\mu*f)(m), & \text{ if $n=m^k$}, \\ 0, & \text{ otherwise}.  
\end{cases} 
\end{equation*}

Hence, for $k\ge 2$, by \eqref{series_g},
\begin{equation*}
\sum_{n_1,\ldots,n_k=1}^{\infty} \frac{(\mu*F)(n_1,\ldots,n_k)}{n_1\cdots n_k} = \sum_{n=1}^{\infty} \frac{g(n)}{n}= \sum_{m=1}^{\infty} \frac{(\mu*f)(m)}{m^k} = \frac1{\zeta(k)} 
\sum_{m=1}^{\infty} \frac{f(m)}{m^k},
\end{equation*}
and the result now follows on applying Theorem \ref{Th_Wintner_F}.
\end{proof}

\begin{proof}[Proof of Theorem {\rm \ref{Th_gcd}}] Consider the function $f(n)=\log n$.
Note that $\mu *\log =\Lambda$ is the von Mangoldt function, 
defined by
\begin{equation*}
\Lambda(n)= 
\begin{cases} \log p, & \text{ if $n=p^\nu$ ($\nu \ge 1$)},\\ 0, & \text{ otherwise}.
\end{cases}    
\end{equation*}

We deduce by identity \eqref{G_form} that
\begin{equation} \label{S_log_k}
S_{\log,k}(x):= \sum_{n_1\cdots n_k\le x} \log (n_1,\ldots, n_k) = \sum_{d^k\delta \le x} \Lambda(d) \tau_k(\delta) =
\sum_{p^{\nu k}\delta \le x} (\log p) \tau_k(\delta).
\end{equation}

We remark that for the functions $\omega(n)$ and $\Omega(n)$ one has 
\begin{equation*}
(\mu * \omega)(n) =
\begin{cases} 1, & \text{ if $n=p$},\\ 0, & \text{ otherwise},
\end{cases}    
\end{equation*}
\begin{equation*}
\sum_{n_1\cdots n_k\le x} \omega ((n_1,\ldots, n_k)) = \sum_{p^k\delta \le x} \tau_k(\delta),
\end{equation*}
respectively
\begin{equation*}
(\mu * \Omega)(n) =
\begin{cases} 1, & \text{ if $n=p^\nu$ ($\nu \ge 1$)},\\ 0, & \text{ otherwise},
\end{cases}    
\end{equation*}
\begin{equation*}
\sum_{n_1\cdots n_k\le x} \Omega ((n_1,\ldots, n_k)) = \sum_{p^{\nu k} \delta \le x} \tau_k(\delta).
\end{equation*}

We present the details of the proof only for the function $f(n)=\log n$.
In the cases $f(n)=\omega(n)$ and $f(n)=\Omega(n)$ the used arguments are similar. 
 
By \eqref{S_log_k} and \eqref{Piltz_form},
\begin{equation*} 
S_{\log,k}(x)= \sum_{p^{\nu} \le x^{1/k}} (\log p) \sum_{\delta \le x/p^{\nu k}}\tau_k(\delta) 
\end{equation*}
\begin{equation} \label{S_log_main}
= \sum_{p^{\nu} \le x^{1/k}} (\log p) \left(\frac{x}{p^{\nu k}} P_{k-1}\Big(\log \frac{x}{p^{\nu k}}\Big) 
+ O\Big(\Big(\frac{x}{p^{\nu k}}\Big)^{\theta_k+\varepsilon}\Big) \right).
\end{equation}

The error term $R_k(x)$ from \eqref{S_log_main} is
\begin{equation*} 
R_k(x) \ll x^{\theta_k+\varepsilon} \sum_{p^{\nu} \le x^{1/k}} \frac{\log p}{p^{\nu k(\theta_k+\varepsilon)}} \ll
x^{\theta_k+\varepsilon} \sum_{p \le x^{1/k}} (\log p) \sum_{\nu=1}^{\infty} \frac1{p^{\nu k(\theta_k+\varepsilon)}} 
\end{equation*}
\begin{equation*}
\ll x^{\theta_k+\varepsilon} \sum_{p \le x^{1/k}} \frac{\log p}{p^{k(\theta_k+\varepsilon)}},
\end{equation*}
that is, by assuming $\theta_k\ge 1/k$,
\begin{equation} \label{R_k}
R_k(x) \ll x^{\theta_k+\varepsilon}.
\end{equation}

Let $P_{k-1}(t)= \sum_{j=0}^{k-1} a_j t^j$. The main term $M_k(x)$ in \eqref{S_log_main} is 
\begin{equation*}
M_k(x)= x \sum_{p^\nu \le x^{1/k}} \frac{\log p}{p^{\nu k}} \sum_{j=0}^{k-1} a_j \left(\log \frac{x}{p^{\nu k}} \right)^j
\end{equation*}
\begin{equation*}
= x \sum_{j=0}^{k-1} a_j \sum_{t=0}^j (-k)^t \binom{j}{t} (\log x)^{j-t} \sum_{p^\nu \le x^{1/k}} \frac{\nu^t (\log p)^{t+1}}{p^{\nu k}},
\end{equation*}
and for any fixed $t$ the inner sum $I_{k,t}(x)$ is, by denoting $m_k=\lfloor \frac{\log x}{k\log p} \rfloor$,
\begin{equation*}
I_{k,t}(x):= \sum_{p^\nu \le x^{1/k}} \frac{\nu^t (\log p)^{t+1}}{p^{\nu k}} =\sum_{p\le x^{1/k}} (\log p)^{t+1} \sum_{\nu =1}^{m_k} \frac{\nu^t}{p^{\nu k}} 
\end{equation*}
\begin{equation*}
= \sum_{p\le x^{1/k}} (\log p)^{t+1} \left( \sum_{\nu =1}^{\infty} \frac{\nu^t}{p^{\nu k}} - \sum_{\nu \ge m_k+1} \frac{\nu^t}{p^{\nu k}}\right) 
\end{equation*}
\begin{equation*}
= \sum_p (\log p)^{t+1} \sum_{\nu =1}^{\infty} \frac{\nu^t}{p^{\nu k}}   
- \sum_{p> x^{1/k}} (\log p)^{t+1} \sum_{\nu =1}^{\infty} \frac{\nu^t}{p^{\nu k}}   
- \sum_{p\le x^{1/k}} (\log p)^{t+1} \sum_{\nu \ge m_k+1} \frac{\nu^t}{p^{\nu k}}. 
\end{equation*}

We note that
\begin{equation} \label{series_id}
\sum_{m=1}^{\infty} m^t x^m = \frac{x}{(1-x)^{t+1}} \psi_{t-1}(x)  \quad (t\in \N, |x|<1),
\end{equation}
where $\psi_{t-1}(x)$ is a polynomial in $x$ of degree $t-1$, more exactly $\psi_{t-1}(x)= 
\sum_{m=0}^{t-1} \left\langle {t \atop m} \right\rangle x^m$. Here $\langle {t \atop m} \rangle $ denote the (classical) Eulerian numbers, 
defined as the number of permutations $h\in S_n$ with $k$ descents 
(a number $i$ is called a descent of $h$ if $h(i) > h(i + 1)$).  
See, e.g., Petersen \cite[Ch.\ 1]{Pet2015}.
Hence for every $t\ge 1$, 
\begin{equation*}
\sum_p (\log p)^{t+1} \sum_{\nu =1}^{\infty} \frac{\nu^t}{p^{\nu k}} = 
\sum_p (\log p)^{t+1} \frac{1/p^k}{(1-1/p^k)^{t+1}} 
\sum_{m=0}^{t-1} \left\langle {t \atop m} \right\rangle \frac1{p^{m k}} 
\end{equation*}
\begin{equation*}
= \sum_{m=0}^{t-1} \left\langle {t \atop m} \right\rangle \sum_p (\log p)^{t+1} \frac{1/p^{k(m+1)}}{(1-1/p^k)^{t+1}}:=b_{k,t},
\end{equation*}
a constant (depending on $k$ and $t$),  since 
\begin{equation*}
\sum_p (\log p)^{t+1} \frac{1/p^{k(m+1)}}{(1-1/p^k)^{t+1}}\le \frac1{(1-1/2^k)^{t+1}}  \sum_p \frac{(\log p)^{t+1}}{p^{k(m+1)}},   
\end{equation*}
and the latter series converges for every $t,m\ge 0$ (since $k\ge 3$).

It is a consequence of \eqref{series_id} that for fixed $t,k\in \N$ we have the estimate
\begin{equation} \label{estimate_1} 
\sum_{\nu=1}^{\infty} \frac{\nu^t}{p^{\nu k}}  \ll \frac1{p^k} \quad \text{ as $p\to \infty$}.
\end{equation}

More generally, for fixed $t,k\in \N$ and $x\ge 1$ real,
\begin{equation} \label{est_gen}
\sum_{\nu \ge x} \frac{\nu ^t}{p^{\nu k}} \ll \frac{x^t}{p^{kx}} \quad \text{ as $p\to \infty$}, 
\end{equation}
uniformly for $p$ and $x$. See \cite[Lemma 4.5]{BorTot2022}, proved by some different arguments.

We also need the estimate 
\begin{equation} \label{estimate_prime}
\sum_{p> x} \frac{(\log p)^\eta}{p^s} \ll \frac{(\log x)^{\eta-1}}{x^{s-1}} \quad \text{ as $x\to \infty$}, 
\end{equation}
where $\eta\ge 0$ and $s>1$ are fixed real numbers. See \cite[Lemma 3.3]{HeyTot2021}.

Using \eqref{estimate_1} and \eqref{estimate_prime} we have
\begin{equation*}
\sum_{p> x^{1/k}} (\log p)^{t+1} \sum_{\nu =1}^{\infty} \frac{\nu^t}{p^{\nu k}} \ll
\sum_{p> x^{1/k}} \frac{(\log p)^{t+1}}{p^k} \ll \frac{(\log x)^t}{x^{1-1/k}}.
\end{equation*}

Furthermore, by \eqref{est_gen},
\begin{equation*}
\sum_{\nu \ge m_k+1} \frac{\nu^t}{p^{\nu k}}\ll \frac{(\log x)^t}{x(\log p)^t},
\end{equation*}
hence
\begin{equation*}
\sum_{p\le x^{1/k}} (\log p)^{t+1} \sum_{\nu \ge m_k+1} \frac{\nu^t}{p^{\nu k}}\ll
\frac{(\log x)^t}{x} \sum_{p\le x^{1/k}} \log p \ll \frac{(\log x)^{t+1}}{x} \sum_{p\le x^{1/k}} 1
\ll \frac{(\log x)^t}{x^{1-1/k}},
\end{equation*}
by using the estimate $\pi(x):= \sum_{p\le x} 1  \ll \frac{x}{\log x}$.

We deduce that
\begin{equation*}
I_{k,t}(x)= b_{k,t} + O(x^{1/k-1} (\log x)^t),     
\end{equation*}
which also holds for $t=0$, and
\begin{equation*}
M_k(x)= xP_{\log,k-1}(\log x) + O(x^{1/k} (\log x)^{k-1}),     
\end{equation*}
where $P_{\log, k-1}(\log x)$ is a polynomial in $\log x$ of degree $k-1$. 

Comparing to \eqref{R_k}, the final error term is $\ll x^{\theta_k+\varepsilon}$.
This finishes the proof. 
\end{proof}

To prove Theorem \ref{Th_class_r} we quote the following Lemma.

\begin{lemma} \label{Lemma_class_f}
If $k\ge 2$ and $f\in {\cal A}_r$ with $r\ge 0$ real, then
\begin{equation*}
\sum_{n_1,\ldots,n_k=1}^{\infty}
\frac{f([n_1,\ldots,n_k])}{n_1^{s_1}\cdots n_k^{s_k}}
= \zeta(s_1-r)\cdots \zeta(s_k-r)
H_{f,k}(s_1,\ldots,s_k),
\end{equation*}
where the multiple Dirichlet series 
\begin{equation*}
H_{f,k}(s_1,\ldots,s_k) = \sum_{n_1,\ldots, n_k=1}^{\infty} \frac{h_{f,k}(n_1,\ldots, n_k)}{n_1^{s_1}\cdots n_k^{s_k}}
\end{equation*}
is absolutely convergent for $\Re s_1,\ldots,\Re s_k > r+1/2$.
\end{lemma}
 
\begin{proof}[Proof of Lemma {\rm \ref{Lemma_class_f}}] 
This is a part of \cite[Lemma.\ 3.1]{HilTot2016}. Note that if $f\in {\cal A}_r$, then the function $f([n_1,\ldots,n_k])$
is multiplicative and its multiple Dirichlet series can be expanded into an Euler product.
\end{proof}

\begin{proof}[Proof of Theorem {\rm \ref{Th_class_r}}]
 By Lemma \ref{Lemma_class_f} we deduce that if $f\in {\cal A}_r$, then
\begin{equation*}
f([n_1,\ldots,n_k]) =\sum_{j_1 d_1=n_1,\ldots,j_k d_k=n_k}
(j_1 \cdots j_k)^r h_{f,k}(d_1,\ldots,d_k),
\end{equation*}
and 
\begin{equation*}
\sum_{n_1\cdots n_k=n} f([n_1,\ldots,n_k]) =  \sum_{j_1d_1 \cdots j_kd_k=n} 
(j_1\cdots j_k)^r h_{f,k}(d_1,\ldots,d_k)
\end{equation*}
\begin{equation*}
= \sum_{jd_1 \cdots d_k=n} h_{f,k}(d_1,\ldots,d_k) j^r \sum_{j_1\cdots j_k=j} 1
= \sum_{jd_1\cdots d_k =n} h_{f,k}(d_1,\ldots,d_k) j^r\tau_{k}(j).
\end{equation*}

Hence
\begin{equation*}
V:= \sum_{n_1\cdots n_k \le x} f([n_1,\ldots,n_k]) =  \sum_{jd_1 \cdots d_k \le x}  h_{f,k}(d_1,\ldots,d_k) j^r\tau_k(j)
\end{equation*}
\begin{equation*}
= \sum_{d_1,\ldots,d_k \le x} h_{f,k}(d_1,\ldots,d_k)  \sum_{j \le x/(d_1\cdots d_k)} j^r\tau_{k}(j).
\end{equation*}

Now by partial summation we deduce from \eqref{Piltz_form} that
\begin{equation*}  
\sum_{n\le x} n^r\tau_k(n) = x^{r+1} T_{k-1}(\log x) + O(x^{r+\theta_k+\varepsilon}),
\end{equation*} 
for every $\varepsilon>0$, where $T_{k-1}(t)$ is a polynomial
in $t$ of degree $k-1$, with leading coefficient $1/(r+1)(k-1)!$. 
This gives that
\begin{equation} \label{est_V}
V=  \sum_{d_1,\ldots,d_k \le x} h_{f,k}(d_1,\ldots,d_k) \Big(\Big(\frac{x}{d_1\cdots d_k}\Big)^{r+1} T_{k-1}\Big(\log \frac{x}{d_1\cdots d_k}\Big) + 
O\Big(\Big(\frac{x}{d_1\cdots d_k}\Big)^{r+\theta_{k} + \varepsilon}\Big) \Big),
\end{equation}
and the error from \eqref{est_V} is, by selecting any $\theta_k\ge 1/2$, 
\begin{equation*}
\ll x^{r+\theta_k +\varepsilon}  \sum_{d_1,\ldots,d_k=1}^{\infty}  \frac{|h_{f,k}(d_1,\ldots,d_k)|}{(d_1\cdots d_k)^{r+\theta_k+\varepsilon}}
\ll x^{r+\theta_k +\varepsilon},
\end{equation*}
where the series converges by Lemma \ref{Lemma_class_f}.

Let $T_{k-1}(t)= \sum_{j=0}^{k-1} b_j t^j$. Then the main term in \eqref{est_V} is
\begin{equation*}
x^{r+1} \sum_{j=0}^{k-1} b_j \sum_{d_1,\ldots,d_k \le x} \frac{h_{f,k}(d_1,\ldots,d_k)}{(d_1\cdots d_k)^{r+1}} \Big(\log \frac{x}{d_1\cdots d_k}\Big)^j
\end{equation*}
\begin{equation*}
= x^{r+1} \sum_{j=0}^{k-1} b_j \sum_{d_1,\ldots,d_k \le x} \frac{h_{f,k}(d_1,\ldots,d_k)}{(d_1\cdots d_k)^{r+1}} \sum_{t=0}^j 
(-1)^t \binom{j}{t} (\log x)^{j-t} (\log (d_1\cdots d_k))^t
\end{equation*}
\begin{equation*}
= x^{r+1} \sum_{j=0}^{k-1} b_j \sum_{t=0}^j (-1)^t \binom{j}{t} (\log x)^{j-t}  
\sum_{d_1,\ldots,d_k \le x} \frac{h_{f,k}(d_1,\ldots,d_k)(\log (d_1\cdots d_k))^t}{(d_1\cdots d_k)^{r+1}}.
\end{equation*}

Write (for a fixed $t$)
\begin{equation*}
\sum_{d_1,\ldots,d_k \le x} \frac{h_{f,k}(d_1,\ldots,d_k)(\log (d_1\cdots d_k))^t}{(d_1\cdots d_k)^{r+1}} 
\end{equation*}
\begin{equation*}
= \left(\sum_{d_1,\ldots,d_k=1}^{\infty} - \sideset{}{'} 
\sum_{d_1,\ldots,d_k}\right)  \frac{h_{f,k}(d_1,\ldots,d_k)(\log (d_1\cdots d_k))^t}{(d_1\cdots d_k)^{r+1}},
\end{equation*}
where $\sum^{'}$ means that $d_1,\ldots,d_k\le x$ does not hold, that is, there exists at least one 
$m$ ($1\le m\le k$) such that $d_m> x$.

Here the multiple series over $d_1,\ldots,d_k$ is convergent by Lemma \ref{Lemma_class_f} and by using that $\log (d_1\cdots d_k)\ll (d_1\cdots d_k)^{\delta}$ for every $\delta>0$. 
Now, to estimate the sum $\sum^{'}$ we can assume, without loss of generality, that $m=1$. 
We obtain that for every $0<\varepsilon <1/2$,
\begin{equation*}
\sideset{}{'}  \sum_{\substack{d_1,\ldots,d_k\\ d_1>x}}  
\frac{|h_{f,k}(d_1,\ldots,d_k)|(\log (d_1\cdots d_k))^t}{(d_1\cdots d_k)^{r+1}} 
\le \sideset{}{'} \sum_{\substack{d_1,\ldots,d_k\\ d_1>x}}  
\frac{|h_{f,k}(d_1,\ldots,d_k)|(d_1\cdots d_k)^{\delta t}}{(d_1\cdots d_k)^{r+1}}\left(\frac{d_1}{x}\right)^{1/2-\varepsilon} 
\end{equation*}
\begin{equation*}
\le x^{\varepsilon -1/2}  \sum_{d_1,\ldots,d_k=1}^{\infty}  
\frac{|h_{f,k}(d_1,\ldots,d_k)|}{d_1^{r+1/2+\varepsilon-\delta t}d_2^{r+1-\delta t} \cdots d_k^{r+1-\delta t}}\ll 
x^{\varepsilon -1/2},
\end{equation*}
if we choose $\delta$ such that $0<\delta < \frac{\varepsilon}{2t}$, where $t\ge 1$ (for $t=0$ one can choose any $\delta>0$), 
since the last series converges by Lemma \ref{Lemma_class_f}. We obtain the final error $x^{r+1} \cdot x^{\varepsilon-1/2}(\log x)^{k-1}$, which is $x^{r+1/2+\varepsilon}$. 
\end{proof}

To prove Theorem \ref{Th_tau} we need the following lemma.

\begin{lemma} \label{Lemma_tau} Let $k\ge 2$. Then
\begin{equation*}
\sum_{n_1,\ldots, n_k=1}^{\infty} \frac{\tau([n_1,\ldots, n_k])}{n_1^{s_1}\cdots n_k^{s_k}} = \zeta^2(s_1)\cdots
\zeta^2(s_k) G_k(s_1,\ldots,s_k),
\end{equation*}
where
\begin{equation*}
G_k(s_1,\ldots,s_k) = \sum_{n_1,\ldots, n_k=1}^{\infty} \frac{g_k(n_1,\ldots, n_k)}{n_1^{s_1}\cdots n_k^{s_k}}
\end{equation*}
is absolutely convergent provided that $\Re s_j >0$ 
\textup{($1\le j\le k$)} and $\Re(s_j+s_{\ell})>1$ 
\textup{($1\le j<\ell \le k$)}.
\end{lemma} 

\begin{proof}[Proof of Lemma {\rm \ref{Lemma_tau}}] 
This is a special case of \cite[Prop.\ 2.3]{TotZha2018}. Note that the function $\tau([n_1,\ldots,n_k])$
is multiplicative and its multiple Dirichlet series can be expanded into an Euler product.
\end{proof}

\begin{proof}[Proof of Theorem {\rm \ref{Th_tau}}]
Similar to the proof of Theorem \ref{Th_class_r}. By Lemma \ref{Lemma_tau} we deduce that
\begin{equation*}
\tau([n_1,\ldots,n_k]) =\sum_{j_1 d_1=n_1,\ldots,j_k d_k=n_k}
\tau(j_1)\cdots \tau(j_k) g_k(d_1,\ldots,d_k),
\end{equation*}
and 
\begin{equation*}
\sum_{n_1\cdots n_k=n} \tau([n_1,\ldots,n_k]) =  \sum_{j_1d_1 \cdots j_kd_k=n} 
\tau(j_1)\cdots \tau(j_k) g_k(d_1,\ldots,d_k)
\end{equation*}
\begin{equation*}
= \sum_{jd_1 \cdots d_k=n} g_k(d_1,\ldots,d_k) \sum_{j_1\cdots j_k=j} \tau(j_1)\cdots \tau(j_k)=
\sum_{jd_1\cdots d_k =n} g_k(d_1,\ldots,d_k) \tau_{2k}(j).
\end{equation*}

Hence
\begin{equation*}
T:= \sum_{n_1\cdots n_k \le x} \tau([n_1,\ldots,n_k]) =  \sum_{jd_1 \cdots d_k \le x}  g_k(d_1,\ldots,d_k)\tau_{2k}(j)
\end{equation*}
\begin{equation*}
= \sum_{d_1,\ldots,d_k \le x} g_k(d_1,\ldots,d_k)  \sum_{j \le x/(d_1\cdots d_k)} \tau_{2k}(j).
\end{equation*}

By applying \eqref{Piltz_form}, we deduce that
\begin{equation} \label{est_T}
T=  \sum_{d_1,\ldots,d_k \le x} g_k(d_1,\ldots,d_k) \Big(\frac{x}{d_1\cdots d_k} P_{2k-1}\Big(\log \frac{x}{d_1\cdots d_k}\Big) + 
O\Big(\Big(\frac{x}{d_1\cdots d_k}\Big)^{\theta_{2k} + \varepsilon}\Big),
\end{equation}
and the error from \eqref{est_T} is 
\begin{equation*}
\ll x^{\theta_{2k}+\varepsilon}  \sum_{d_1,\ldots,d_k \le x} \frac{|g_k(d_1,\ldots,d_k)|}{(d_1\cdots d_k)^{\theta_{2k}+\varepsilon}}
\ll x^{\theta_{2k}+\varepsilon},
\end{equation*}
assuming that $\theta_{2k}\ge 1/2$ and by using Lemma \ref{Lemma_tau}.

Let $P_{2k-1}(t)= \sum_{j=0}^{2k-1} d_j t^j$. Then the main term in \eqref{est_T} is
\begin{equation*}
x \sum_{j=0}^{2k-1} d_j \sum_{d_1,\ldots,d_k \le x} \frac{g_k(d_1,\ldots,d_k)}{d_1\cdots d_k} \Big(\log \frac{x}{d_1\cdots d_k}\Big)^j
\end{equation*}
\begin{equation*}
= x \sum_{j=0}^{2k-1} d_j \sum_{d_1,\ldots,d_k \le x} \frac{g_k(d_1,\ldots,d_k)}{d_1\cdots d_k} \sum_{t=0}^j 
(-1)^t \binom{j}{t} (\log x)^{j-t} (\log (d_1\cdots d_k))^t
\end{equation*}
\begin{equation*}
= x \sum_{j=0}^{2k-1} d_j \sum_{t=0}^j (-1)^t \binom{j}{t} (\log x)^{j-t}  
\sum_{d_1,\ldots,d_k \le x} \frac{g_k(d_1,\ldots,d_k)(\log (d_1\cdots d_k))^t}{d_1\cdots d_k}.
\end{equation*}

Write (for a fixed $t$),
\begin{equation*}
\sum_{d_1,\ldots,d_k \le x} \frac{g_k(d_1,\ldots,d_k)(\log (d_1\cdots d_k))^t}{d_1\cdots d_k} =
\left(\sum_{d_1,\ldots,d_k=1}^{\infty} - \sideset{}{'} 
\sum_{d_1,\ldots,d_k}\right)  \frac{g_k(d_1,\ldots,d_k)(\log (d_1\cdots d_k))^t}{d_1\cdots d_k} 
\end{equation*}
where $\sum^{'}$ means that $d_1,\ldots,d_k\le x$ does not hold, that is, there exists at least one 
$m$ ($1\le m\le k$) such that $d_m> x$.

Here the multiple series over $d_1,\ldots,d_k$ is convergent by Lemma \ref{Lemma_tau} and by $\log (d_1\cdots d_k)\ll (d_1\cdots d_k)^{\delta}$ for every $\delta>0$. Now, to estimate the sum $\sum^{'}$ we can assume, without loss of generality, that $m=1$. 
We obtain that for every $0<\varepsilon <1$,
\begin{equation*}
\sideset{}{'}  \sum_{\substack{d_1,\ldots,d_k\\ d_1>x}}  
\frac{|g_k(d_1,\ldots,d_k)|(\log (d_1\cdots d_k))^t}{d_1\cdots d_k} 
\le \sideset{}{'} \sum_{\substack{d_1,\ldots,d_k\\ d_1>x}}  
\frac{|g_k(d_1,\ldots,d_k)|((d_1\cdots d_k)^{\delta t}}{d_1\cdots d_k}\left(\frac{d_1}{x}\right)^{1-\varepsilon} 
\end{equation*}
\begin{equation*}
\le x^{\varepsilon -1}  \sum_{d_1,\ldots,d_k=1}^{\infty}  
\frac{|g_k(d_1,\ldots,d_k)|}{d_1^{\varepsilon-\delta t}d_2^{1-\delta t} \cdots d_k^{1-\delta t}}  \ll 
x^{\varepsilon -1} 
\end{equation*}
if we choose $\delta$ such that $0<\delta < \frac{\varepsilon}{2t}$, where $t\ge 1$ (for $t=0$ one can choose any $\delta>0$), 
since the last series converges by Lemma \ref{Lemma_tau}. 
We obtain the final error $x\cdot x^{\varepsilon-1}(\log x)^{2k-1}$, which is of order $x^{\varepsilon}$. 
\end{proof}

\end{document}